\documentclass[11pt]{article}

\usepackage{amsmath}
\usepackage{amssymb}
\usepackage{bbm}

\newtheorem{theorem}{Theorem}[section]
\newtheorem{proposition}[theorem]{Proposition}
\newtheorem{lemma}[theorem]{Lemma}
\newtheorem{corollary}[theorem]{Corollary}

\newenvironment{proof}{\medbreak\noindent{\it Proof.}\rm}{\hfill$\square$\rm}

\numberwithin{equation}{section}

\renewcommand{\Re}{\operatorname{Re}}
\renewcommand{\Im}{\operatorname{Im}}

\newcommand{\Rn}{{ \mathbb R}^n}
\newcommand{\C}{{\mathbb  C}}
\newcommand{\D}{{\mathbb  D}}
\newcommand{\Cn}{{\mathbb  C}^n}
\newcommand{\Rnm}{{\mathbb R}_-^n}
\newcommand{\Rnp}{{\mathbb R}_+^n}

\newcommand{\cF}{{\mathcal F}}
\newcommand{\cO}{{\mathcal O}}
\newcommand{\E}{{\mathcal E}}
\newcommand{\cL}{{\mathcal  L}}

\newcommand{\PSH}{{\operatorname{PSH}}}

\newcommand{\Capa}{{\operatorname{Cap}\,}}
\newcommand{\Log}{{\operatorname{Log}\,}}

\newcommand{\Vol}{{\operatorname{Vol}}}
\newcommand{\Covol}{{\operatorname{Covol}}}

\begin{document}

\begin{center}{\Large\bf Plurisubharmonic geodesics and interpolating sets}
\end{center}

\begin{center}{\large Dario Cordero-Erausquin and Alexander Rashkovskii}
\end{center}

\bigskip

\begin{abstract} We apply a notion of geodesics of plurisubharmonic functions to interpolation of compact subsets of $\Cn$. Namely, two non-pluripolar, polynomially closed, compact subsets of $\Cn$  are interpolated as level sets $L_t=\{z: u_t(z)=-1\}$ for the geodesic $u_t$ between their relative extremal functions with respect to any ambient bounded domain. The sets $L_t$ are described in terms of certain holomorphic hulls. In the toric case, it is shown that the relative Monge-Amp\`ere capacities of $L_t$ satisfy a dual Brunn-Minkowski inequality.
\end{abstract}

\section{Introduction}
In the classical complex interpolation theory of Banach spaces, originated by Calder\'{o}n \cite{Ca} (see \cite{BL} and, for more recent developments, \cite{CEK} and references therein), a given family of Banach spaces $X_\xi$ parameterized by boundary points $\xi$ of a domain $C\subset\mathbb C^N$ gives rise to a family of Banach spaces $X_\zeta$ for all $\zeta\in C$. A basic setting is interpolation of two spaces, $X_0$ and $X_1$, for a partition $\{C_0, C_1\}$ of $\partial C$. More specifically, one can take $C$ to be the strip $0<{\Re}\, \zeta <1$ in the complex plane and $C_0, C_1$ the corresponding boundary lines, then the interpolated norms depend only on  $t=\Im\zeta$. In the finite dimensional case $X_j=(\Cn,\|\cdot\|_j)$, $j=0,1$, they
are defined in terms of the family of mappings $C\to\Cn$, bounded and analytic in the strip, continuous up to the boundary and tending to zero as ${\Im}\, \zeta\to\infty$, see details in \cite{BL}.
In this setting, the volume of the unit ball $B_t$ of $(\Cn,\|\cdot\|_t)$, $0<t<1$, was proved in \cite{CE} to be a logarithmically concave function of $t$.

When the given norms $\|\cdot\|_j$ on $\Cn$ are toric, i.e., satisfy $\|(z_1,\ldots,z_n)\|_j=\|(|z_1|,\ldots,|z_n|)\|_j$, the interpolated norms are toric as well and the balls $B_t$ are Reinhardt domains of $\Cn$ obtained as the multiplicative combinations (geometric means) of the balls $B_0$ and $B_1$. The logarithmic concavity implies that volumes of the multiplicative combinations
\begin{equation}\label{geommean} K_t^\times = K_0^{1-t}\,K_1^t\subset\Rn
\end{equation}
of any two convex bounded neighbourhoods $K_0$ and $K_1$ of the origin in $\Rn$ satisfy the Brunn-Minkowski inequality
\begin{equation}\label{volBM} \Vol(K_t^\times)\ge \Vol(K_0)^{1-t} \Vol(K_1)^t,\quad 0<t<1.\end{equation}
Note also that in \cite{S2}--\cite{S4}, the interpolated spaces were related to convex hulls and complex geodesics with convex fibers. In particular, it put the interpolation in the context of analytic multifunctions.

In this note, we develop a slightly different -- albeit close -- approach to the interpolation of compact, polynomially convex subsets of $\Cn$ by sets arising from a notion of plurisubharmonic geodesics. The technique originates from results on geodesics in the spaces of metrics on compact K\"{a}hler manifolds due to Mabuchi, Semmes, Donaldson, Berndtsson and others (see \cite{G12} and the bibliography therein). Its local counterpart for plurisubharmonic functions from Cegrell classes on domains of $\Cn$ was introduced in \cite{BB} and \cite{R17a}. We will need here a special case when the geodesics can be described as follows.

Let
$$A=\{\zeta\in\C:\:0< \log|\zeta| < 1\}$$
be the annulus bounded by the circles
$$A_j=\{\zeta:\: \log|\zeta|=j\},\quad j=0,1.$$
Given two plurisubharmonic functions $u_0$ and $u_1$ in a bounded hyperconvex domain $\Omega\subset\Cn$, equal to zero on $\partial \Omega$, we
consider the class $W(u_0,u_1)$ of all plurisubharmonic functions $u(z,\zeta)$ in the product domain $\Omega\times A$, such that $\limsup u(z,\zeta)\le u_j(z)$ for all $z\in\Omega$ as $\zeta\to A_j$. The function
\begin{equation}\label{upenv} \widehat u(z,\zeta)=\sup\{u(z,\zeta):\: u\in W(u_0,u_1)\} \end{equation}
belongs to the class and satisfies $\widehat u(z,\zeta)=\widehat u(z,|\zeta|)$, which gives rise to the functions $u_t(z):=\widehat u(z,e^t)$, $0<t<1$, the {\it geodesic} between $u_0$ and $u_1$. When the functions $u_j$ are bounded, the geodesic $u_t$ tend to $u_j$ as $t\to j$, uniformly on $\Omega$. One of the main properties of the geodesics is that they linearize the energy functional
\begin{equation}\label{enfunc}  \E(u)=\int_\Omega u(dd^c u)^n, \end{equation}
see \cite{BB}, \cite{R17a} (where actually more general classes of plurisubharmonic functions are considered).

Given two non-pluripolar compact sets $K_0,K_1\subset\Cn$, let $u_j$ denote the relative extremal functions of $K_j$, $j=0,1$, with respect to a bounded hyperconvex neighbourhood $\Omega$ of $K_0\cup K_1$, i.e.,
\begin{equation}\label{initial}
u_j(z)=\omega_{K_j,\Omega}(z)=\limsup_{y\to z}\, \sup\{u(y):\: u\in  \PSH_-(\Omega),\ u|_{K_j}\le -1\},
\end{equation}
where $\PSH_-(\Omega)$ is the collection of all nonpositive plurisubharmonic functions in $\Omega$. The functions $u_j $ belong to $\PSH_-(\Omega)$ and satisfy $(dd^cu_j)^n=0$ on $\Omega\setminus K_j$, see \cite{Kl}.

Assume, in addition, each  $K_j$ to be polynomially convex (in the sense that it coincides with its polynomial hull). This implies $\omega_{K_j,\Omega'}=-1$ on $K_j$ for some (and thus any) bounded hyperconvex neighborhood $\Omega'$ of $K_j$ and that $\omega_{K_j,\Omega'}\in C(\overline{\Omega'})$. In particular, the functions $u_j=-1$ on $K_j$ and are continuous on $\overline\Omega$. The geodesics $u_t$ converge to $u_j$ uniformly as $t\to j$ \cite{R17a} and so, by the Walsh theorem, the upper envelope $ \widehat u(z,\zeta)$ (\ref{upenv}) is continuous on $\Omega\times A$, which, in turn, implies  $u_t\in C(\overline\Omega\times [0,1])$.

As was shown in \cite{R17b}, functions $u_t$ in general are different from the relative extremal functions of any subsets of $\Omega$. Consider nevertheless the sets where they attain their minimal value, $-1$:
\begin{equation}\label{L_t} L_t=\{z\in \Omega:\: u_t(z)=-1\},\quad  0< t< 1.\end{equation}
By the continuity of the geodesic at the endpoints, the sets $L_t$ converge (say, in the Hausdorff metric) to $ K_j$ when  $t\to j\in\{0,1\}$ and so, they can be viewed as interpolations between $K_0$ and $K_1$.

The curve $t\mapsto L_t$ can be in a natural way identified with the multifunction $\zeta\mapsto K_\zeta:=L_{\log|\zeta|}$. Note however that it is not an {\it analytic multifunction} (for the definition, see, e.g., \cite{O}, \cite{S1}, \cite{Po}) because its graph
$\{(z,\zeta)\in \Omega\times A:\: \widehat u(z,\zeta)=-1\}$ is not pseudoconcave.

In Section~2, we show that the interpolating sets $L_t$ can be represented as sections $K_t=\{z:\: (z,e^t)\in \widehat K\} $ of the holomorphic hull
$\widehat K$ of the set
\begin{equation}\label{hull}K^A:=(K_0\times A_0) \cup (K_1\times A_1)\subset \C^{n+1}\end{equation}
with respect to functions holomorphic in $\Cn\times (\C\setminus\{0\})$.

In Section~3, we study the relative Monge-Amp\`ere capacities $\Capa(L_t,\Omega)$ of the sets $L_t$; recall that for $K\Subset\Omega$,
$$\Capa(K,\Omega)= \sup\{(dd^c u)^n(K):\: u\in\PSH_-(\Omega),\ u|_K\le -1\}=(dd^c\omega_{K,\Omega})^n(\Omega), $$
see \cite{Kl}.
 It was shown in \cite{R17a} that the function $t\mapsto\Capa(L_t,\Omega)$  is convex,
  which was achieved by using linearity of the energy functional (\ref{enfunc})
  along the geodesics.
In the case when $\Omega$ is the unit polydisk $\D^n$ and $K_j$ are Reinhardt sets, the convexity of the Monge-Amp\`ere capacities was rewritten in \cite{R17b} as convexity of covolumes of certain unbounded convex subsets $P_t$ of the positive orthant $\Rnp$ (that is, volumes of their complements to $\Rnp$).
Here, we use a convex geometry technique to prove Theorem~\ref{BMcovol} stating that actually the covolumes of the sets $P_t$ are {\sl logarithmically convex}.
Since in this case the sets $L_t$ are exactly the geometric means $K_t^\times$ of $K_0$ and $K_1$, this implies the dual Brunn-Minkowski inequality for their Monge-Amp\`ere capacities,
\begin{equation}\label{logconvcap}\Capa(K_t^\times,\D^n)\le \Capa(K_0,\D^n)^{1-t} \Capa(K_1,\D^n)^t,\quad 0<t<1.\end{equation}
In addition, an equality here occurs for some $t\in (0,1)$ if and only if $K_0=K_1$.

It is quite interesting that the volume of $K_t^\times$ satisfies the opposite Brunn-Minkowski inequality (\ref{volBM}), i.e., it is logarithmically {\sl concave}. Furthermore, so are the standard logarithmic capacity in the complex plane and the Newtonian capacity in $\Rn$ with respect to the Minkowski addition \cite{B1}, \cite{B2}, \cite{Ran}. The difference here is that the relative Monge-Amp\`ere capacity is, contrary to the logarithmic or Newton capacities, a local notion,  which leads to the dual Brunn-Minkowski inequality  (\ref{logconvcap}), exactly like for the covolumes of coconvex bodies \cite{KT}.

A natural question that remains open is to know whether the logarithmic convexity of the relative Monge-Amp\`ere capacities is also true in the general, non-toric case. No non-trivial examples of (\ref{logconvcap}) in this setting are known so far.

\section{Level sets as holomorphic hulls}\label{sect:holhulls}

Let $K_0,K_1$ be two non-pluripolar compact subsets of a bounded hyperconvex domain $\Omega\subset\Cn$, and let $L_t=L_{t,\Omega}$ be the interpolating sets defined by (\ref{L_t}) for the geodesic $u_t=u_{t,\Omega}$ with the endpoints $u_j=\omega_{K_j,\Omega}$. We start with an observation that if the sets $K_j$ are polynomially convex, then the sets $L_t$
are actually independent of the choice of the domain $\Omega$ containing $K_0\cup K_1$.

\begin{lemma}\label{indep} If $\Omega'$ and $\Omega''$ are bounded hyperconvex neighborhoods of non-pluripolar, polynomially convex, compact sets $K_0$ and $K_1$, then $L_{t,\Omega'}=L_{t,\Omega''}$.
\end{lemma}

\begin{proof} By the monotonicity of $\Omega\mapsto u_{t,\Omega}$, it suffices to show the equality for $\Omega'\Subset\Omega''$.
Since $u_{t,\Omega''}\le u_{t,\Omega'}$, the inclusion $L_{t,\Omega'}\subset L_{t,\Omega''}$ is evident.
Denote now
$$\delta=-\inf\{u_{j,\Omega''}(z):\: z\in\partial\Omega',\ j=0,1\}\in(0,1).$$
Recall that the geodesics $u_{t,\Omega}$ come from the maximal plurisubharmonic functions $\widehat u_{\Omega}$ in $\Omega\times A$ for the annulus $A$ bounded by the circles $A_j$ where $\log|\zeta|=j$.
Then the function
$$\hat v:=\frac1{1-\delta}(\widehat u_{\Omega''}+\delta)\in\PSH(\Omega'\times A)\cap C(\overline{\Omega'\times A})$$
satisfies $(dd^c\hat v)^{n+1}=0$ in $\Omega'\times A$ and
\begin{equation}\label{bdv}\lim \hat v(z,\zeta)= -1\ {\rm \ as\ } (z,\zeta)\to K_j\times A_j.\end{equation}
Moreover, since $\hat v\ge 0$ on $\partial \Omega'\times A$ and its restriction to each $A_j$ satisfies $(dd^c\hat v)^{n}=0$ on $A_j\setminus K_j$, the boundary conditions (\ref{bdv}) imply
$$ \lim \hat v(z,\zeta)\ge u_{j,\Omega'}\ {\rm as\ } \zeta\to A_j.$$
Therefore, we have $\hat v\ge \hat u$ in the whole $\Omega'\times A$.
 If $z\in L_{t,\Omega''}$, this gives us $-1\ge u_{t,\Omega'}(z)$ and so, $z\in L_{t,\Omega'}$, which completes the proof.
\end{proof}

\medskip

Next step is comparing the sets $L_t$   with other interpolating sets, $K_t$, defined as follows.
Set
\begin{equation}\label{hatK} \widehat K=\widehat K(\Omega)=\{(z,\zeta)\in \Omega\times A:\: u(z,\zeta)\le M(u)\ \forall u\in \PSH_-(\Omega\times A)\},\end{equation}
where $M(u)=\max_j M_j(u)$ and
$$ M_j(u)=\limsup\,u(z,\zeta)\ {\rm as \ }(z,\zeta)\to K_j\times A_j, \quad j=0,1.$$
Note that the set $\widehat K$ will not change if one replaces $\PSH_-(\Omega\times A)$ by the collection of all bounded from above (or just bounded) plurisubharmonic functions on $\Omega\times A$.

Denote by $\widehat K_\zeta$ the section of $\widehat K$ over $\zeta\in A$:
$$\widehat K_\zeta=\widehat K_\zeta(\Omega) = \{z\in\Omega:\: (z,\zeta)\in \widehat K\},\quad \zeta\in A.$$
The set $\widehat K$ is invariant under rotation of the $\zeta$-variable, so $\widehat K_\zeta$ depends only on $|\zeta|$.
We set then
$$K_t=\widehat K_{e^t},\quad 0<t<1.$$

\begin{theorem}\label{prop1}If $K_j$ are non-pluripolar, polynomially convex compact subsets of $\Omega$, then $L_t=K_t$ for all $0<t<1$. \end{theorem}

\begin{proof}
First, we prove the inclusion
\begin{equation} \label{incl1} L_t \subset K_t, \end{equation}
that is,
\begin{equation} \label{Mbound} u(z,t)\le M(u)\quad \forall z\in L_t\end{equation}
for all $u\in\PSH_-(\Omega\times A)$.
By the scalings $u\mapsto c\, u$, we can assume that
$u-\min_j M_j(u)\le 1$ on $\Omega\times A$.
Then the function
$$\phi(z,\zeta)=u(z,\zeta)- (1-\log|\zeta|)M_0(u) -(\log|\zeta|) M_1(u)-1$$
belongs to $\PSH_-(\Omega\times A)$ and
$$\limsup\phi(z,\zeta)\le -1 {\rm\ as \ }(z,\zeta)\to K_j\times A_j.$$
In other words, $\phi_t(z):=\phi(z,e^t)$ is a subgeodesic for $u_0$ and $u_1$, so $\phi_t\le u_t$. Therefore, $\phi_t\le -1$ on $L_t$, which gives us (\ref{Mbound}).

To get the reverse inclusion, assume $z\in K_t$. Then, by definition of $\widehat K$, we get $u_t(z)\le M(u_t)=-1$ and, since $u_t\ge -1$ everywhere, $u_t(z)=-1$.
\end{proof}

\medskip

The set $\widehat K$ can actually be represented as a holomorphic hull of the set
$$K^A=(K_0\times A_0) \cup (K_1\times A_1),$$ which is similar to what one gets in the classical interpolation theory. This can be concluded by standard arguments relating plurisubharmonic and holomorphic hulls (see, e.g., \cite{Range}).

\begin{proposition}\label{loc} Let $K_0,K_1$ be two non-pluripolar, polynomially convex compact subsets of a bounded hyperconvex domain $\Omega\subset\Cn$. Then the set $\widehat K$ defined by (\ref{hatK}) is the holomorphic hull of the set $K^A$
with respect to the collection of all functions holomorphic on $\Omega\times\C_*$ (here $\C_*=\C\setminus\{0\}$).
\end{proposition}

\begin{proof} The domain $\Omega\times\C_*$ is pseudoconvex, so it suffices to show that $\widehat K$ is the
$\cF$-hull $\widehat K_\cF$ of the set (\ref{hull}) with respect to $\cF=\PSH(\Omega\times\C_*)$.

Take any hyperconvex domain $\Omega'$ such that $K_0\cup K_1\subset \Omega'\Subset\Omega$.
Since $\cF$ forms a subset of the collection of all bounded from above psh functions on $\Omega'\times A$, we have $\widehat K':=\widehat K(\Omega')\subset \widehat K_\cF$. Moreover, by Lemma~\ref{indep} and Theorem~\ref{prop1}, we have $\widehat K'=\widehat K$, which implies $\widehat K\subset \widehat K_\cF$.

  Let $u_t$ be the geodesic of $u_0,u_1$ in $\Omega$. Then its psh image $\hat u(z,\zeta)$  can be extended to $\Omega\times \C_*$ as $\hat U(z,\zeta)=u_0(z)-\log|\zeta|$ for $-\infty<\log|\zeta|\le 0$ and $\hat U(z,\zeta)=u_1(z)+\log|\zeta|-1$ for $1\le\log|\zeta|< \infty$. Indeed, the function
$$\hat v(z,\zeta) =\max\{u_0(z)-\log|\zeta|, u_1(z)+\log|\zeta|-1\}$$
is psh on $\Omega\times A$, continuous on $\Omega\times \overline A$, and equal to $u_j$ on $\Omega\times A_j$. Therefore, it coincides with $\hat U$ on $\Omega\times(\C_*\setminus A)$, Since $\hat v\le \hat u$ on $\Omega\times A$ and $\hat v=\hat u$ on $\Omega\times \partial A$, the claim is proved.

Let $(z^*,\zeta^*)\in \widehat K_\cF$. By the definition of $\widehat K_\cF$, since $\hat U\in\PSH(\Omega\times\C_*)$,
$$\hat u(z^*,\zeta^*)=\hat U(z^*,\zeta^*)\le \sup \{\hat U(z,\zeta):\: (z,\zeta)\in K^A\}=-1,$$
so $z^*\in \widehat K_t$ with $t=\log|\zeta^*|$.
\end{proof}

\medskip
Finally, since the sets $L_t$ are independent of the choice of $\Omega$, we get the following description of the interpolated sets $K_t$.

\begin{corollary}\label{cor} Let $K_0,K_1$ be two non-pluripolar, polynomially convex compact subsets of $ \Cn$ and let $\Omega$ be a bounded hyperconvex domain containing $K_0\cup K_1$. Denote by $u_t$ the geodesic of the functions $u_j=\omega_{K_j,\Omega}$, $j=0,1$. Then for any $\zeta\in A$,
$$K_t=\{z\in\Omega:\: u_{t}(z)=-1\}=\{z\in\Cn: |f(z,\zeta)|\le\|f\|_{K^A}\ \forall f\in\cO(\Cn\times \C_*)\}$$
with $t=\log|\zeta|$.
\end{corollary}

{\it Remark.} Note that the considered hulls are taken with respect to functions holomorphic in $\Cn\times \C_*$ and not in $\C^{n+1}$ (that is, not the polynomial hulls). This reflects the fact that in the definition of $K^A$, the circles $A_0$ and $A_1$ may be interchanged.
Since for any polynomial $P(z,\zeta)$ and any $\zeta$ inside the disc $|\zeta|<e$, we have $|P(z,\zeta)|\le \max
\{|P(z,\xi)|: \: |\xi|=e\}$, each section of the {\sl polynomial} hull of $K^A$ must contain $K_1$, so such a hull does not provide any interpolation between $K_0$ and $K_1$.

\section{Log-convexity of Monge-Amp\`ere capacities}

Let, as before, $K_0$ and $K_1$ be non-pluripolar, polynomially convex compact subsets of a bounded hyperconvex domain $ \Omega\Subset\Cn$, $u_t$ be the geodesic between $u_j= \omega_{K_j,\Omega}$, and let $K_t$  be the corresponding interpolating sets as described in Section~\ref{sect:holhulls}. As was mentioned, their relative Monge-Amp\`ere capacities satisfy the inequality
$$\Capa(K_t,\Omega)\le(1-t)\,\Capa(K_0,\Omega) +t\,\Capa(K_1,\Omega).$$

Let now $\Omega=\D^n$ and assume the sets $K_j$ to be Reinhardt. The polynomial convexity of $K_j$ means then that their logarithmic images
$$Q_j=\Log K_j=\{s\in\Rnm:\: (e^{s_1},\ldots,e^{s_n})\in K_j\}$$
are complete convex subsets of $\Rnm$, i.e., $Q_j+\Rnm\subset\Rnm$; when this is the case, we will also say that $K_j$ is complete logarithmically convex. In this situation, the sets $K_t$ are, as in the classical interpolation theory, the geometric means of $K_j$. Note however that our approach extends the classical -- convex -- setting to a wider one.

\begin{proposition} The interpolating sets $K_t$ of two non-pluripolar, complete logarithmically convex, compact Reinhardt sets $K_0, K_1\subset \D^n$ coincide with
$$K_t^\times:=K_0^{1-t}K_1^t=\{z:\: |z_l|=|\eta_l|^{1-t} |\xi_l|^{t}, \ 1\le l\le n,\ \eta\in K_0,\ \xi\in K_1\}.$$
\end{proposition}

\begin{proof} We prove this by using the representation of the sets $K_t$ as $L_t=\{z:\: u_t(z)=-1\}$ and a formula for the geodesics in terms of the Legendre transform \cite{Gu}, \cite{R17a}.  By and large, this is Calder\'{o}n's  method.

As was noted in \cite[Thm. 4.3]{R17a}, the inclusion $K_t^\times\subset L_t$ follows from convexity of the function $\check u_t(s)=u_t (e^{s_1},\ldots,e^{s_n})$ in $(s,t)\in \Rnm\times (0,1)$ since $s\in \log K_t^\times$ implies $\check u_t(s)\le -1$. To prove the reverse inclusion, we use a representation for $\check u_t$ given by \cite[Thm. 6.1]{R17b}:
$$ \check u_t=\cL[(1-t)\max\{h_{Q_0}+1,0\} + t \max\{h_{Q_1}+1,0\}],\quad 0<t<1,$$
where
$$ \cL[f](y)=\sup_{x\in\Rn}\{\langle x,y\rangle -f(x)\}$$
is the {\it Legendre transform} of $f$,
$$h_Q(a)=\sup_{s\in Q} \langle a,s\rangle,\quad a\in\Rnp
$$
is the support function of a convex set $Q\subset\Rnm$, and
$Q_j=\Log K_j$.

Let $z\notin K_t^\times$, then one can assume that none of its coordinates equals zero, so the corresponding point $\xi=(\log|z_1|,\ldots,\log|z_n|)\in\Rnm$ does not belong to $Q_t=(1-t)Q_0+tQ_1$. Therefore there exists $b\in\Rnp$ such that
$$ \langle b,\xi\rangle > h_{Q_t}(b)=(1-t)h_{Q_0}(b)+ t\, h_{Q_1}(b);$$
by the homogeneity, one can assume $h_{Q_0}(b),h_{Q_1}(b)>-1$ as well. Then
\begin{eqnarray*}
\check u_t(\xi) &=& \sup_{a\in\Rnp}[\langle a,\xi\rangle - (1-t)\max\{h_{Q_0}(a)+1,0\} - t \max\{h_{Q_1}(a)+1,0\}]\\
&>& (1-t)[h_{Q_0}(b)-(h_{Q_0}(b)+1)] + t [h_{Q_1}(b)-(h_{Q_1}(b)+1)]=1,
\end{eqnarray*}
so $\xi$ does not belong to $\Log L_t$ and consequently $z\notin L_t$.
\end{proof}

\medskip

The crucial point for the Reinhardt case is a formula from \cite[Thm. 7]{ARZ} (see also \cite{R17b}) for the Monge-Amp\`ere capacity of complete logarithmically convex compact sets $K\subset\D^n$:
$$\Capa (K,\D^n)=n!\,\Covol(Q^\circ):=n!\, \Vol({\Rnp\setminus Q^\circ}),$$
where
$$Q^\circ=\{x\in\Rnp: \langle x,y\rangle \le -1 \ \forall y\in Q\}$$
is the {\it copolar} to the set $Q=\Log K$. In particular,
\begin{equation}\label{capakt}\Capa(K_t)=n!\,\Covol(Q_t^\circ)\end{equation}
for the copolar $Q_t^\circ$ of the set $Q_t=(1-t)Q_0+t\,Q_1$.

\begin{proposition}\label{BMcovol} We have
\begin{equation}\label{BMQ}\Covol(Q_t^\circ)\le \Covol(Q_0^\circ)^{1-t}\,\Covol(Q_1^\circ)^t,\quad 0<t<1.\end{equation}
If an equality here occurs for some $t\in (0,1)$, then $Q_0=Q_1$.
\end{proposition}

\begin{proof} Let, as before, $h_Q$ be the restriction of the support function of a convex set $Q\subset\Rnm$ to $\Rnp$:
$$ h_Q(x)=\sup\{\langle x,y\rangle:\: y\in Q\},\ x\in\Rnp.$$
We have then
\begin{eqnarray*}
\int_{\Rnp} e^{h_{Q_t}(x)}\,dx &=& \int_{\Rnp} dx \int_{-h_{Q_t}(x)}^\infty e^{-s}\,ds = \int_0^\infty e^{-s}\,ds \int_{h_{Q_t}(x)\ge -s} dx \\
&=& \int_0^\infty e^{-s} \Vol(\{h_{Q_t}(x)\ge -s\})\,ds \\
&=& \Vol(\{h_{Q_t}(x)\ge -1\}) \int_0^\infty e^{-s}s^n \,ds\\
&=& n!\, \Covol({Q_t^\circ}).
\end{eqnarray*}
Note that $h_{Q_t}=(1-t) h_{Q_0}+ t \, h_{Q_1}$.
Therefore, by H\"{o}lder's inequality with $p=(1-t)^{-1}$ and $q=t^{-1}$, we have
\begin{equation}\label{Hold}
\int_{\Rnp} e^{h_{Q_t}(x)}\,dx\le\left(\int_{\Rnp} e^{h_{Q_0}(x)}dx\right)^{1-t}\left(\int_{\Rnp} e^{h_{Q_1}(x)}dx\right)^{t},
\end{equation}
which proves (\ref{BMQ}).

An equality in (\ref{BMQ}) implies the equality case in H\"{o}lder's inequality (\ref{Hold}), which means the functions
$e^{h_{Q_0}}$ and $e^{h_{Q_1}}$ are proportional, so $h_{Q_0}(x)=h_{Q_1}(x)+C$ for all $x\in\Rnp$. Since both $h_{Q_0}(x)$ and $h_{Q_1}(x)$ converge to $0$ as $x\to 0$ along $\Rnp$, we get $C=0$, which completes the proof.
\end{proof}

\medskip
Finally, by (\ref{capakt}), we get

\begin{theorem}\label{CapBM} For polynomially convex, non-pluripolar compact Reinhardt sets $K_j\Subset\D^n$, the Monge-Amp\`ere capacity $\Capa(K_t,\D^n)$ is a logarithmically convex function of $\,t$; in other words, the Brunn-Minkowski inequality (\ref{logconvcap}) holds. An equality in (\ref{logconvcap}) occurs for some $t\in (0,1)$ if and only if $K_0=K_1$.
\end{theorem}

{\it Remark.} The general situation of compact, polynomially convex Reinhardt sets reduces to the case $K_0,K_1\subset\D^n$ because for $K$ in the polydisk $\D_R^n$ of radius $R$, we have $\Capa(K,\D_R^n)=\Capa(\frac1R K,\D^n)$ and $(\frac1R K)_t = \frac1R K_t$.

\medskip
{\bf Acknowledgement.} Part of the work was done while the second named author was visiting Universit\'e Pierre et Marie Curie in March 2017; he thanks Institut de Math\'ematiques de Jussieu for the hospitality. The authors are grateful to the anonymous referee for careful reading of the text.

\bigskip\noindent
{\sc Dario Cordero-Erausquin}

\noindent 
Institut de Math\'ematiques de Jussieu, Sorbonne Universit\'e, 4 place Jissieu, 75252 Paris Cedex 05, France

\noindent
\emph{e-mail:} dario.cordero@imj-prg.fr

\medskip\noindent
{\sc Alexander Rashkovskii}

\noindent 
Tek/Nat, University of Stavanger, 4036 Stavanger, Norway

\noindent
\emph{e-mail:} alexander.rashkovskii@uis.no

\end{document}